\documentclass{amsart}
\usepackage{amssymb,amsthm,amstext,amsmath,amscd,latexsym,amsfonts}
\usepackage[all]{xy}
\def\ulL{{\underline{L}}}
\def\ulM{{\underline{M}}}
\def\ulN{{\underline{N}}}
\def\ulZ{{\underline{Z}}}
\def\ulX{{\underline{X}}}
\def\ulY{{\underline{Y}}}
\def\ulW{{\underline{W}}}
\def\ulV{{\underline{V}}}
\def\ulU{{\underline{U}}}

\def\ARSX{\Sigma _X}
\def\ARTX{{\underline{\Sigma _X}}}
\def\ulSigma{{\underline{\Sigma}}}

\def\ulf{{\underline{f}}}
\def\ulg{{\underline{g}}}
\def\ulh{{\underline{h}}}
\def\ulw{{\underline{w}}}

\def\uldeltaMN{\delta _{\ulM , \ulN }}

\def\SCM{{\underline{\mathrm{CM}}}}

\def\SCMR{{\underline{\mathrm{CM}}}(R)}
\def\SHom{{\underline{\mathrm{Hom}}}}
\def\SHomR{{\underline{\mathrm{Hom}}_R}}
\def\SEnd{{\underline{\mathrm{End}}}}

\def\homo{\leq _{\underline{hom}}}

\def\sto{\leq _{st}}
\def\trio{\leq _{tri}}

\def\CMR{\mathrm{CM}(R)}
\def\indCMR{\mathrm{indCM}(R)}

\def\modC{\mathrm{mod}(\CMR)}
\def\SmodC{{\underline{\mathrm{mod}}}(\CMR)}
\def\GR{\mathrm{G}(\CMR )}
\def\GSR{\mathrm{G}(\SCMR )}

\def\Hom{\mathrm{Hom}}
\def\HomR{\mathrm{Hom}_R}
\def\End{\mathrm{End}}
\def\Ext{\mathrm{Ext}}
\def\Im{\mathrm{Im}}

\def\Z{\mathbb Z}
\def\m{\mathfrak m}
\def\p{\mathfrak p}

\def\T{\mathcal T}

\def\FF{\mathcal F} 
\def\GG{\mathcal G} 
\def\II{\mathcal I} 
\def\JJ{\mathcal J} 
\def\KK{\mathcal K} 
\newtheorem{theorem}{Theorem}[section]
\newtheorem{lemma}[theorem]{Lemma}
\newtheorem{corollary}[theorem]{Corollary}
\newtheorem{proposition}[theorem]{Proposition}
\theoremstyle{definition}
\newtheorem{definition}[theorem]{Definition}
\newtheorem{example}[theorem]{Example}
\theoremstyle{remark}
\newtheorem{remark}[theorem]{Remark}
\numberwithin{equation}{section}

\begin{document}

\title{On the stable hom relation and stable degenerations of Cohen-Macaulay modules}
\author{Naoya Hiramatsu}
\dedicatory{Dedicated to Professor Yuji Yoshino on the occasion of his sixtieth birthday.} 
\address{Department of general education, National Institute of Technology, Kure College, 2-2-11, Agaminami, Kure Hiroshima, 737-8506 Japan}
\email{hiramatsu@kure-nct.ac.jp}
\thanks{This work was supported by JSPS Grant-in-Aid for Young Scientists (B) 15K17527. }
\date{\today}

\begin{abstract}
We study the stable hom relation for Cohen-Macaulay modules over Gorenstein local algebras. 
We give the sufficient condition to make the stable hom relation a partial order when the base algebra is of finite representation type. 
As an application, we give the description of stable degenerations of Cohen-Macaulay modules over simple singularities of several types by using the stable hom relation. 
\end{abstract}

\maketitle

\section{Introduction} 
In ring theory, the hom relation is a basic relation for finitely generated modules over finite dimensional algebras \cite{A82, AR85, R86, Z99, Y02, JSZ05}. 
Let $k$ be a field and $R$ a $k$-algebra. 
The relation is defined by a dimension of a hom-set between finitely generated modules as a $k$-module, that is, we define the relation $M \leq _{hom} N$ by a relation $\dim _k \Hom _R (X, M) \leq \dim _k \Hom_R (X, N)$ for each finitely generated modules $X$.  
Auslander-Reiten \cite{AR85} use the relation to investigate when indecomposable modules are determined by the composition factors. 
In \cite{Z99}, Zwara gave a characterization of degenerations of modules over representation finite algebras in relation with the hom relation. 
We remark that the hom relation is not always a partial order. 
It has been studied by many authors \cite{A82, AR85, B89, Y02} when the relation is actually a partial order.   

In the paper, we investigate the hom relation on a stable category of Cohen-Macaulay modules $\SCMR$ over (not necessary Artinian) Gorenstein $k$-algebra.  
First we compare the Auslander-Reiten theory on $\CMR$ with that on $\SCMR$.  
We look into the relation between AR sequences and AR triangles of Cohen-Macaulay modules (Proposition \ref{AR triangle}). 
We consider a relation on $\SCMR$ which is the stable analogue of the hom relation (Definition \ref{dfn of hom order}) and shall show that it is actually a partial order if the algebra is of finite representation type with certain assumptions (Theorem \ref{cd}). 
In Section \ref{stable degenerations of Cohen-Macaulay modules}, we attempt to characterize the stable degenerations of Cohen-Macaulay modules by using the stable hom relation. 
The concept of stable degenerations of Cohen-Macaulay modules was introduced by Yoshino \cite{Y11}.  
It is closely related to ordinary degenerations of modules \cite{HY13, SZ14}.  
We shall show that the stable degenerations over several simple singularities can be controlled by the stable hom relation (Theorem \ref{Main Theorem}). 
To show this, we use the stable analogue of the argument over finite dimensional algebras in \cite{Z99}. 
As a conclusion, we give the description of stable degenerations of Cohen-Macaulay modules over simple singularities of type ($A_n $) (Theorem \ref{simple singularity}). 

The stable hom relation has been studied by Auslander-Reiten \cite{AR85} and they also considered when the relation is a partial order. 
But the techniques in this paper are different from them because they used the fact that the ordinary (not stable) hom relation is a partial order. 
In our setting, the hom-set does not always have finite dimension, so that we can not apply their argument. 

\section{Stable hom relation on Cohen-Macaulay modules}\label{stable hom relation}

Throughout the paper $R$ is a commutative complete Gorenstein local $k$-algebra where $k$ is an algebraically closed field of characteristic $0$. 
For a finitely generated $R$-module $M$, we say that $M$ is a Cohen-Macaulay $R$-module if 
$$
\Ext_{R}^{i} (M, R) = 0 \quad \text{for any } i > 0.  
$$
We denote by $\CMR$ the category of Cohen-Macaulay $R$-modules with all $R$-homomorphisms.
We also denote by $\SCMR$ the stable category of $\CMR$. 
The objects of $\SCMR$ are the same as those of $\CMR$, and the morphisms of $\SCMR$ are elements of $\SHom _R(M, N) = \Hom _R(M, N)/P(M, N)$ for $M, N \in \SCMR$, where $P(M, N)$ denote the set of morphisms from $M$ to $N$ factoring through free $R$-modules. 
We write $\SHom _R (M, N)$ for $\Hom _{\SCMR} (M, N)$. 
For a Cohen-Macaulay module $M$, denote by $\ulM$ to indicate that it is an object of $\SCMR$. 
For a finitely generated $R$-module $M$, take a free resolution 
$$
\cdots \rightarrow F_1 \xrightarrow{d} F_0 \rightarrow M \rightarrow 0.
$$
We denote $\Im (d)$ by $\Omega M$. 
We also denote $\mathrm{Coker} (\HomR (d, R))$ by $\mathrm{Tr} M$, which is called an Auslander transpose of $M$. 
We note that the functor $\Omega$ defines a functor giving an auto-equivalence on $\SCMR$.
It is known that $\SCMR$ has a structure of a triangulated category with the suspension functor defined by the functor $\Omega$.  
See \cite[Chapter 1]{H} for details. 
Since $R$ is Gorenstein, by the definition of a triangle, $ \ulL \to \ulM \to \ulN \to \ulL [1]$ is a triangle in $\SCMR$ if and only if there is an exact sequence $0 \to L \to M' \to N \to 0$  in  $\CMR$ with $\ulM ' \cong \ulM$ in $\SCMR$, that is, $M'$ is isomorphic to $M$ up to free summand. 
Since $R$ is complete, $\CMR$, hence $\SCMR$, is a Krull-Schmidt category, namely each object can be decomposed into indecomposable objects up to isomorphism uniquely.

In the paper we use the theory of Auslander-Reiten (abbr. AR) sequences and triangles of Cohen-Macaulay modules. 
Let us recall the definitions of those notions. 
See \cite{Y} for AR sequences and \cite{H, RV02} for AR triangles.

\begin{definition} 
Let $X$, $Y$ and $Z$ be Cohen-Macaulay $R$-modules.  
\begin{itemize}
\item[(1)] A short exact sequence $\ARSX : 0 \to Z \to Y \xrightarrow{f} X \to 0 $ is said to be an AR sequence ending in $X$ (or starting from $Z$) if it satisfies
\begin{itemize}
\item[(AR1)] $X$ and $Z$ are indecomposable.  

\item[(AR2)] $\ARSX$ is not split. 

\item[(AR3)] If $g : W \to X$ is not a split epimorphism, then there exists $h : W \to Y$ such that $g = f \circ h$.  
\end{itemize}
\item[(2)] We also say that a triangle $\ARTX : \ulZ \to \ulY \xrightarrow{\ulf} \ulX \xrightarrow{\ulw} \ulZ [1] $ is an AR triangle ending in $\ulX$ (or starting from $\ulZ$) if it satisfies
\begin{itemize}
\item[(ART1)] $\ulX$ and $\ulZ$ are indecomposable. 

\item[(ART2)] $\ulw \not=0$. 

\item[(ART3)] If $\ulg : \ulW \to \ulX$ is not a split epimorphism, then there exists $\ulh : \ulW \to \ulY$ such that $\ulg = \ulf \circ \ulh$.  
\end{itemize}
\end{itemize}
\end{definition}  

\begin{proposition}\label{AR triangle}
Let $\ARSX:   0 \to Z \to Y \xrightarrow{f} X \to 0 $ be an AR sequence ending in $X$. 
Then $\ARTX: \ulZ \to \ulY \xrightarrow{\ulf} \ulX \xrightarrow{\ulw} \ulZ [1] $ is an AR triangle ending in $\ulX$. 
\end{proposition}

\begin{proof}
We shall show that $\ulSigma$ satisfies (ART1), (ART2) and (ART3). 

\noindent
(ART1) It is obvious. 

\noindent
(ART2) If $\ulw$ is zero, then $\ulSigma$ is split. 
Thus there exists $\ulg : \ulX \to \ulY$ such that $\ulf \circ \ulg = 1_{\ulX}$. 
Note that $f \circ g \in \mathrm{rad} \End _R (X)$ since $\Sigma$ is not split. 
It yields that $\ulf \circ \ulg = 1_{\ulX} \in \mathrm{rad} \SEnd _R (X)$. 
This is a contradiction and $\ulw$ must be non zero.   

\noindent
(ART3) Let $\ulg : \ulW \to \ulX$ be not a split epimorphism. 
Then $g : W \to X$ is also not a split epimorphism. 
By (AR3), we have a morphism $h : W \to Y$ such that $f \circ h = g$. 
We conclude that $\underline{f} \circ \underline{h} = \ulg$.  
\end{proof}

We say that $\CMR$ (resp. $\SCMR$) admits AR sequences (resp. AR triangles) if there exists an AR sequence (resp. AR triangle) ending in $X$ (resp. $\ulX$) for each indecomposable Cohen-Macaulay $R$-module $X$ which is not free. 
We also say that $(R, \m)$ is an isolated singularity if each localization $R_{\p }$ is regular for each prime ideal $\p$ with $\p \not= \m$. 
If $R$ is an isolated singularity, $\CMR$ admits AR sequences (cf. \cite[Theorem 3.2]{Y}). 
As a corollary of Proposition \ref{AR triangle}, $\SCMR$ admits AR triangles if $R$ is an isolated singularity.

\begin{corollary}\label{correspondence}
If $R$ is an isolated singularity, we have an 1-1 correspondence between the set of isomorphism classes of AR sequences in $\CMR$ and that of AR triangles in $\SCMR$. 
\end{corollary}

\begin{proof}
According to Proposition \ref{AR triangle}, we can define the mapping from the set of AR sequences to the set of AR triangles. 
Note that AR triangles (resp. AR sequences) ending in $\ulX$ (resp. $X$) are unique up to isomorphism of triangles (resp. sequences) for a given indecomposable $\ulX$ (resp. $X$) (see \cite{H, RV02}). 
Hence it follows from Proposition \ref{AR triangle} that the mapping is surjective. 
The injectivity of the mapping is clear. 
\end{proof}
By virtue of the lemma below, we see that $\SHom _R (M, N)$ has finite dimension over $k$ for $M$, $N \in \CMR$ if $R$ is an isolated singularity.

\begin{lemma}\label{stable hom}\cite[Lemma 3.9]{Y}
Let $M$ and $N$ be finitely generated $R$-modules. 
Then we have a functorial isomorphism
$$
\SHom _R (M, N) \cong \mathrm{Tor}_1 ^R (\mathrm{Tr} M, N).
$$ 
\end{lemma}

In what follows, we always assume that $R$ is an isolated singularity, and then the following definition makes sense.

\begin{definition}\label{dfn of hom order}
For $M$, $N \in \CMR$ we define $\ulM \homo \ulN$ if $[\ulX, \ulM] \leq [\ulX, \ulN]$ for each $\ulX \in \SCMR$. 
Here $[\ulX, \ulM] $ is an abbreviation of $\dim_k \SHom _R (X, M)$.  
\end{definition}

For an AR triangle $\ulZ \to \ulY \to \ulX \to \ulZ [1]$, we denote $\ulZ$ (resp. $\ulX$) by $\tau \ulX$ (resp. $\tau ^{-1} \ulZ$). 
For an AR sequence $0 \to Z \to Y \to X \to 0$, $Z$ (resp. $X$) is also denoted by $\tau X$ (resp. $\tau ^{-1} Z$) (see \cite[Definition 2.8]{Y}.). 
By Proposition \ref{AR triangle}, $\tau \ulX \cong \underline{\tau X}$ for each indecomposable Cohen-Macaulay $R$-module $X$.

\begin{remark}\label{Serre functor}
Reiten and Van den Bergh \cite{RV02} show that a Hom-finite $k$-linear triangulated category $\T$ admits AR triangles if and only if $\T$ has a Serre functor. 
We can show that $\SCMR$ is also a Hom-finite triangulated category which has a Serre functor if $R$ is an isolated singularity. 
Actually $\SCMR$ has a Serre functor $\underline{\tau (- )}[1]$ (cf. \cite[Lemma 3.10]{Y} ). 
Note that $\underline{\Omega X} \cong \ulX [-1]$ and $\underline{\tau X} \cong \underline{\Omega ^{2-d} X} \cong \ulX [d-2]$ where $d = \dim R$. 
Hence we have $\underline{\tau (- )}[1] \cong (-)[d-1]$. 
See also \cite[Corollary 2.5.]{IT10}. 
\end{remark}

Now let us consider the full subcategory of the functor category of $\CMR$ which is called the Auslander category. 
We give a brief review of the Auslander category (see \cite[Chapter 4 and 13]{Y} for details). 
The Auslander category $\modC$ is the category whose objects are finitely presented contravariant additive functors from $\CMR$  to the category of abelian groups and whose morphisms are natural transformations between functors. 
The following lemma is a key of our result in this section. 
For an additive subcategory $\mathcal{A}$ of an abelian category, which is skeletally small and closed under extensions, we denote by $K_0 (\mathcal{A} )$ the Grothendieck group of $\mathcal{A}$.

\begin{lemma}\cite[Theorem 13.7]{Y}\label{cartan map}
The group homomorphism 
$$
\gamma : \GR \to K_{0}(\modC ),  
$$
defined by $\gamma (M) =[ \Hom _R ( \ ,M)]$  for $M \in \CMR$, is injective. 
Here $\GR$ is a free abelian group $\bigoplus \ \Z \cdot X$, where $X$ runs through all isomorphism classes of indecomposable objects in $\CMR$.  
\end{lemma}

We denote by $\SmodC$ the full subcategory $\modC$ consisting of functors $F$ with $F(R )= 0$. 
Note that every object $F \in \SmodC$ is obtained from a short exact sequence in $\CMR$. 
Namely we have the short exact sequence $0 \to L \to M \to N \to 0$ such that 
$$
0 \to \Hom _R (\  , L) \to \Hom _R (\  , M) \to \Hom _R (\  , N) \to F \to 0 
$$
is exact in $\modC$. 
Since $F \in \SmodC$ is a subfunctor of $\Ext _R ^1 (\  , L)$ for some $L \in \CMR$, $F(X)$ has finite length for each $X \in \CMR$ if $R$ is an isolated singularity.  
Therefore we can define a group homomorphism associated with $X$ in $\CMR$
\begin{equation}\label{length hom}
\varphi _X : K_{0}(\SmodC ) \to \Z \ ;  \quad [F] \mapsto \dim _k F(X).
\end{equation}
If  $0 \to Z \to Y \to X \to 0$  is an AR sequence in $\CMR$, then the functor $S_X$  defined by an exact sequence 
$$
0 \to \Hom _R (\  , Z) \to \Hom _R (\  , Y) \to \Hom _R (\  , X) \to S_X \to 0 
$$
is a simple object in $\modC$  and all the simple objects in $\modC$ are obtained in this way from AR sequences.

We say that $R$ is of finite representation type if there are only a finite number of isomorphism classes of indecomposable Cohen-Macaulay $R$-modules. 
We note that if $R$ is of finite representation type, then $R$ is an isolated singularity (cf. \cite[Chapter 3.]{Y}). 
It is proved in \cite[(13.7.4)]{Y} that 
\medskip

{\it  for each object $F$ in $\SmodC$, there is a filtration by subobjects  $0 \subset F_1 \subset F_2 \subset \cdots \subset F_n = F$ such that each $F_{i}/F_{i-1}$ is a simple object in $\modC$ if $R$ is of finite representation type. }
\medskip

\noindent
We also remark that, since $\CMR$ is a Krull-Schmidt category,   
$$
S_X (Y) = 
\begin{cases}
k \quad \text{if} \ X \cong Y, \\
0 \quad \text{if} \ X \not\cong Y.
\end{cases}
$$
for an indecomposable module $Y \in \CMR$. See \cite[(4.11)]{Y} for instance.

\begin{lemma}\label{filtration}
If $R$ is of finite representation type, then we have the equality in $K_{0}(\SmodC )$
$$
[\SHom _R ( -, M )] = \sum _{X_i \in \indCMR} [\underline{X_i}, \ulM] \cdot [S_{X_i}]
$$
for each $M \in \CMR$. 
\end{lemma}

\begin{proof}
For $F = \SHom _R ( -, M )$, $F(R )= 0$, so that $F \in \SmodC$. 
Since $R$ is of finite representation type, $F$ has a filtration by simple objects $S_{X_i}$. 
Hence we have the equality in $K_{0}(\SmodC )$:
$$
[F ] =  \sum _{X_i \in \indCMR} c_i \cdot [S_{X_i}]. 
$$
By using the homomorphism in (\ref{length hom}), we see that 
$$
[\underline{X_j}, \ulM ] = \varphi _{X_j} ([F]) = \dim _k \sum _{X_i \in \indCMR} c_i \cdot \dim_k S_{X_i}(X_j ) = c_j.
$$
Therefore we obtain the equation in the lemma.  
\end{proof}

\begin{theorem}\label{cd}
Let $R$ be of finite representation type and $M$ and $N$ be Cohen-Macaulay $R$-modules. 
Suppose that $[\ulX , \ulM] = [\ulX, \ulN]$ for each $\ulX \in \SCMR$. 
Then $\ulM \oplus \underline{\Omega M} \cong \ulN \oplus \underline{\Omega N}$. 
\end{theorem}

\begin{proof}
Under the circumstances, we see that $[\SHom _R ( -, M )] = [\SHom _R (-, N )]$ in $K_{0}(\SmodC )$, hence in $K_{0}(\modC )$. 
Note that $\SHom _R (- , M ) \cong \Ext ^1 _R (-, \Omega M)$ for each $M \in \CMR$ (cf. \cite{Buc}). 
We have the resolution in $\modC$: 
$$
0 \to \Hom _R ( -, \Omega M ) \to \Hom _R ( -, P_M) \to \Hom _R ( -, M) \to \SHom _R (-, M ) \to 0,  
$$
where $P_M$ is a free $R$-module. 
Thus we have 
$$
 [\Hom _R ( -, M)] + [\Hom _R ( -, \Omega M )] - [\Hom _R ( -, P_M)] = [\Hom _R (-, N)] + [\Hom _R (-, \Omega N )] - [\Hom _R ( -, P_N )].  
$$
Hence, 
$$
 [\Hom _R ( -, M)] + [\Hom _R ( -, \Omega M )] + [\Hom _R ( -, P_N)] = [\Hom _R ( -, N)] + [\Hom _R ( -, \Omega N )] + [\Hom _R ( -, P_M )]. 
$$
According to Lemma \ref{cartan map}, we get 
$$
M \oplus \Omega M \oplus P_N \cong N \oplus \Omega N \oplus P_M.  
$$
Therefore $\ulM \oplus \underline{\Omega M} \cong \ulN \oplus \underline{\Omega N}$. 
\end{proof}

As a corollary, we have the following.   

\begin{corollary}\label{stable hom order}
Let $R$ be of finite representation type and $M$ and $N$ be Cohen-Macaulay $R$-modules. 
Suppose that $\ulU \cong \ulU [-1]$ for each indecomposable Cohen-Macaulay $R$-module $U$. 
Then $[\ulX , \ulM] = [\ulX, \ulN]$ for each $\ulX \in \SCMR$ if and only if $\ulM \cong \ulN$. 
In partiqular, $\homo$ is a partial order on $\SCMR$. 
\end{corollary}

\begin{example}\label{stable hom order on A_even}
Let $R$ be a one dimensional simple singularity of type $(A_n )$, that is $R = k[[x, y]]/(x^{n+1} + y^2)$. 
If $n$ is an even integer, one can show that $X$ is isomorphic to $\Omega X$ up to a free summand for each $X \in \CMR$, so that $\ulX \cong \ulX[-1]$. See \cite[Proposition 5.11]{Y}. 
Thus $\homo$ is a partial order on $\SCMR$ if $n$ is an even integer. 
\end{example}

On the above example, if $n$ is an odd integer, we have indecomposable modules $X \in \CMR$ such that $\ulX \not \cong \ulX[-1]$. 
In fact, let $N_\pm = R/(x^{(n+1)/2} \pm \sqrt{-1}y)$. 
Then $N_{+}$ (resp. $N_{-}$) is a Cohen-Macaulay $R$-module which is isomorphic to $\Omega N_-$ (resp. $\Omega N_+$), so that $\underline{N_+} \not\cong \underline{N_+}[-1]$ (resp. $\underline{N_-} \not\cong \underline{N_-}[-1]$).
Though we can also show that $\homo$ is a partial order on $\SCMR$ if $n$ is an odd integer.

\begin{proposition}\label{stable hom order on A_n}
Let $R = k[[x, y]]/(x^{n+1} + y^2)$. 
Then $[\ulX , \ulM] = [\ulX, \ulN]$ for each $\ulX \in \SCMR$ if and only if $\ulM \cong \ulN$. 
\end{proposition}

\begin{proof}
We show the case when $n$ is an odd integer.  
Let $I_i = (x^i , y)$ be ideals of $R$ for $1 \leq i \leq (n-1)/2$. 
Then $\{ I_1, \cdots I_{(n-1)/2}, N_{+}, N_{-}  \}$ is a complete list of non free indecomposable Cohen-Macaulay $R$-modules.   
Note that $I_i \cong \Omega I_i$ up to free summand for $i = 1, \cdots (n-1)/2$. 
See \cite[Paragraph (9.9)]{Y}. 
Now let $M$, $N \in \CMR$ and suppose that  $[\ulX , \ulM] = [\ulX, \ulN]$ for each $\ulX \in \SCMR$. 
Set $\ulM = \bigoplus _{i=1}^{(n-1)/2} \underline{I_i}^{m_i} \oplus \underline{N_+}^{m_+} \oplus \underline{N_-}^{m_-}$ and $\ulN = \bigoplus _{i=1}^{(n-1)/2} \underline{I_i}^{n_i} \oplus \underline{N_+}^{n_+} \oplus \underline{N_-}^{n_-}$.
By Theorem \ref{cd}, $\ulM \oplus \underline{\Omega M} \cong \ulN \oplus \underline{\Omega N}$ . 
Thus
$$
\bigoplus _{i=1}^{(n-1)/2} \underline{I_i}^{2m_i} \oplus \underline{N_+}^{m_+ + m_-} \oplus \underline{N_-}^{m_+ + m_-} \cong \bigoplus _{i=1}^{(n-1)/2} \underline{I_i}^{2n_i} \oplus \underline{N_+}^{n_+ + n_-} \oplus \underline{N_-}^{n_+ + n_-}.
$$
Hence we have equalities:
$$
m_i = n_i , \quad m_+ + m_- = n_+ + n_-.
$$
Here we remark that 
$$
\SHom _R (N_\pm, N_\mp) \cong \Ext_R ^1 (N_\pm, N_\pm) = 0.  
$$
Using this, we have 
$$
[\underline{N_+}, \ulM] = \sum_{i=1}^{(n-1)/2} m_i [\underline{N_+}, \underline{I_i}] + m_+  [\underline{N_+}, \underline{N_+}]  = \sum_{i=1}^{(n-1)/2} n_i [\underline{N_+}, \underline{I_i}] + n_+  [\underline{N_+}, \underline{N_+}] = [\underline{N_+}, \ulN]. 
$$
This equality show that $m_+ = n_+$, so that $m_- = n_-$. Consequently $\ulM \cong \ulN$. 
\end{proof}

\begin{remark}\label{rmk1 of stable hom order}
The stable hom relation $\homo$ is not always a partial order on $\SCMR$ even if the base ring $R$ is a simple singularity of type ($A_n$). 
Let $R = k[[x, y, z]]/(x^3 + y^2 + z^2)$, that is, $R$ is a two dimensional simple singularity of type $(A_2 )$. 
And let $I$ (resp. $J$) be an ideal generated by $(x, y)$ (resp. $(x^2, y)$). 
Note that the set $\{ I, J \}$ is a complete list of non free indecomposable Cohen-Macaulay $R$-modules (see \cite[Chapter 10]{Y} for instance). 
Then it is easy to see that $[\underline{I}, \underline{I}] = [\underline{I}, \underline{J}] = [\underline{J}, \underline{I}] = [\underline{J}, \underline{J}]= 1$. 
However $\underline{I} \not \cong \underline{J}$. Thus $\homo$ is not a partial order on $\SCMR$.
\end{remark}

When $Z$ is indecomposable then denote by $\mu (\ulZ, \ulM)$ the multiplicity of $\ulZ$ as a direct summand of $\ulM$. 
On an AR triangle, we have the following.

\begin{proposition}\label{prop of AR triangles}
Let $\ARTX: \ulZ \to \ulY \to \ulX \to \ulZ [1]$ be an AR triangle in $\SCMR$. 
Then the following statements hold for each indecomposable $U \in \CMR$. 
\begin{itemize}
\item[(1)] $[\ulU, \ulX ] + [\ulU, \ulZ ] - [\ulU, \ulY ] = \mu (\ulU, \ulX) + \mu (\ulU, \ulX [-1])$. 

\item[(2)] If $\ulU$ is periodic of period 2, that is, $\ulU \cong \ulU [2]$, then
$$
[\ulU, \ulX ] + [\ulU, \ulZ ] - [\ulU, \ulY ] = [\ulU [-1] , \ulX ] + [\ulU [-1], \ulZ ] - [\ulU [-1], \ulY ] 
$$ 
\end{itemize}
\end{proposition}

\begin{proof}
(1) Apply $\SHom _R (U, - )$ to the triangle $\ARTX$, we have a long exact sequence as follows: 
$$
\begin{CD} 
\SHom _R (U, Z [-1])@>>> \SHom _R (U , Y [-1]) @>{\SHom _R(U, g [-1])}>> \SHom _R (U , X [-1] ) @>>>  \\
\SHom _R (U, Z )@>{\SHom _R(U, f)}>> \SHom _R (U , Y )@>{\SHom _R(M , g)}>> \SHom _R (U , X ) @>>> \\
\SHom _R (U, Z [1])@>{\SHom _R(U , f [1])}>> \SHom _R (U , Y [1] ) @>>> \SHom _R (U , X [1]) @>>> .\\
\end{CD} 
$$
Since $\SEnd _R (X )/ \mathrm{rad} \SEnd _R (X ) \cong k$ for each non free indecomposable Cohen-Macaulay module $X$ and by the property of an AR triangle (ART3), we have
$$
\mathrm{Ker} \ \SHom _R (U ,  f [i+1] ) \cong \mathrm{Coker}\  \SHom _R (U , g [i] ) \cong k^{\mu (\ulU, \ulX [i])}
$$
for all $i \in \Z$. 
In particular, the following sequence is exact. 
$$
0 \to k^{\mu (\ulU, \ulX[-1] )} \to \SHomR (U , Z) \to \SHomR (U , Y ) \to \SHomR (U , X ) \to k^{\mu (\ulU, \ulX)} \to 0.
$$
Therefore we obtain the required equation.

(2) Note from (1) that the equations
$$
[\ulU, \ulX ] + [\ulU, \ulZ ] - [\ulU, \ulY ] = \mu (\ulU, \ulX) + \mu (\ulU, \ulX [-1])
$$
and 
$$
[\ulU [-1] , \ulX ] + [\ulU [-1], \ulZ ] - [\ulU [-1], \ulY ] = \mu (\ulU [-1], \ulX) + \mu (\ulU [-1], \ulX [-1])
$$
hold. 
Since the shift functor $(-) [1]$ (hence $(-) [-1]$) is an auto-functor, $\mu (\ulU , \ulX) = \mu (\ulU [-1] , \ulX [-1] )$. 
Moreover $\mu (\ulU [-1] , \ulX) = \mu (\ulU  , \ulX [-1] )$ for $\ulU [-1] [-1] \cong \ulU [-2] \cong \ulU$. 
Consequently we get the assertion. 
\end{proof}

\begin{remark}\label{rmk2 of stable hom order}
By using Proposition \ref{prop of AR triangles}, one can show that $\homo$ is a partial order on $\SCMR$ if $\ulU \cong \ulU [-1]$ for each $U \in \CMR$ without the assumption that $R$ is of finite representation type.  
\end{remark}

\section{Stable hom relation and Grothendieck group of $\SCMR$}\label{Hom order and Grothendieck group}

For later reference we state some results on the stable hom relation between modules which give the same class in the Grothendieck group of $\SCMR$.  
The Grothendieck group of $\SCMR$ (more generally a triangulated category) is defined by 
$$
K_0 (\SCMR ) = \GSR /< \ulX + \ulZ - \ulY \ | \text{There is a triangle }\ulZ \to \ulY \to \ulX \to \ulZ [1] \text{ in } \SCMR >, 
$$
where $\GSR$ is a free abelian group $\bigoplus _{\ulX \in \mathrm{ind} \SCMR} \Z \cdot \ulX$. 
We refer the reader to \cite[Chapter 3]{H} for the details. 
Since $R$ is Gorenstein, one can show that $[\ulM ] = [\ulN ]$ in $K_0 (\SCMR )$ if and only if $[M \oplus P] = [N \oplus Q]$ in $K_0 (\CMR )$ for some free $R$-modules $P$, $Q$.

\begin{lemma}\label{stable analogue of Yoshino's theorem}
If $R$ is of finite representation type, we have the equality of subgroups of $\GSR$: 
$$
\begin{array}{l}
< \ulX + \ulZ - \ulY \ | \text{There is a triangle }\ulZ \to \ulY \to \ulX \to \ulZ [1] \text{ in } \SCMR > \\
= < \ulX + \ulZ - \ulY \ | \text{There is an AR triangle }\ulZ \to \ulY \to \ulX \to \ulZ [1] \text{ in } \SCMR  >. 
\end{array}
$$
\end{lemma}

\begin{proof}
It follows from Corollary \ref{correspondence}, \cite[Theorem 13.7]{Y} and the definition of triangles in stable categories.  
\end{proof}

Let $R$ be of finite representation type. 
In the rest of the paper, we assume that the nonisomorphic indecomposable Cohen-Macaulay $R$-modules are indexed by $\{ 1, 2, \dots, n \}$, say $X_1$, $\dots$, $X_n$. 
By Lemma \ref{stable analogue of Yoshino's theorem}, $[\ulM ] = [\ulN ]$ in $K_0 (\SCMR)$ if and only if there exist AR triangles $\underline{Z_k} \to \underline{Y_k} \to \underline{X _k} \to \underline{Z_k}[1]$ and non-zero integers $b_k$ such that  the equality 
$$
[\ulN ] - [\ulM ] = \sum _{k \in \KK} b_k \cdot ([\underline{X_k} ]+ [\underline{Z_k}] - [\underline{Y _k}])
$$
holds for some $\KK \subseteq \{ 1, 2, \dots, n \}$. 
We write positive and negative coefficients separately. 
That is,  we express the equality as 
\begin{equation}\label{equation in Grothendieck group}
[\ulN ] - [\ulM ] = \sum _{i \in \II} c_i \cdot ([\underline{X_i} ]+ [\underline{Z_i}] - [\underline{Y _i}]) - \sum _{j \in \JJ} d _{j} \cdot ([\underline{X_{j}}]  + [\underline{Z_{j}}] - [\underline{Y _{j}}] ),  
\end{equation} 
where $\II$, $\JJ \subseteq \{ 1, 2, \dots, n \}$ are disjoint sets and $c_i$, $d_j$ are non-negative integers. 
Note that $X_i \not \cong X_j$ if $i \in \II$ and $j \in \JJ$.

Now we consider the following condition:

\begin{itemize}
\item[$(\ast )$] For an AR triangle $\ulZ \to \ulY \to \ulX \to \ulZ [1]$, $[\ulX ] + [\ulZ] - [\ulY] = [\ulX [-1]] + [\ulZ [-1]] - [\ulY [-1]]$ in $\GSR$.    
\end{itemize}

\begin{remark}\label{the condition for Grothendieck group}
\begin{itemize}
\item[(1)] The condition $(\ast )$ says that, for AR triangles $\ulZ \to \ulY \to \ulX \to \ulZ [1]$ and $\ulZ' \to \ulY '\to \ulX' \to \ulZ' [1]$, $[\ulX ] + [\ulZ] - [\ulY] = [\ulX'] + [\ulZ'] - [\ulY']$ in $\GSR$ if $[\ulU, \ulX ]+ [\ulU, \ulZ] - [\ulU, \ulY ] = [\ulU, \ulX' ]+ [\ulU, \ulZ'] - [\ulU, \ulY' ]$ for each indecomposable $\ulU$. 
Since it follows from Proposition \ref{prop of AR triangles} that $\mu (\ulX , \ulX) + \mu (\ulX , \ulX[-1]) = \mu (\ulX , \ulX') + \mu (\ulX , \ulX'[-1])$, we see that $\ulX \cong \ulX'$ or $\ulX \cong \ulX' [-1]$. 
If $\ulX \cong \ulX'$ the equation is obvious and assume that $\ulX \cong \ulX' [-1]$. 
Then by ($\ast$), 
$$
\begin{array}{ll}
[\ulX ] + [\ulZ] - [\ulY] &= [\ulX' [-1] ] + [\ulZ' [-1] ] - [\ulY' [-1] ] \\
&= [\ulX' ] + [\ulZ' ] - [\ulY' ] \\
\end{array}
$$
in $\GSR$. 

\item[(2)] We also remark that, under the condition $(\ast )$, $\tau \ulX \cong \ulX [-1]$ holds if $\ulX \cong \ulX [2]$. 
Suppose that $\ulX \cong \ulX [-1]$. The claim follows from Remak \ref{Serre functor}. 
Suppose that $\ulX \not\cong \ulX [-1]$. 
For the AR triangle $\ARTX :\tau \ulX \to \underline{E_X} \to \ulX \to \tau \ulX [1]$, we have
$$
\ulX \oplus \tau \ulX \oplus \underline{E_X}[-1] \cong \ulX [-1] \oplus \tau \ulX [-1]  \oplus \underline{E_X}.  
$$
Assume that $\tau \ulX \not\cong \ulX [-1]$. 
Then $\tau \ulX [-1] \not\cong \ulX [-2] \cong \ulX$. 
Since $\ulX \not \cong \ulX [-1]$, $\tau \ulX \not \cong \tau \ulX [-1]$.  
Then $\ulX \oplus \tau \ulX $ is ismorphic to a direct summand of $\underline{E_X}$, so that $\underline{E_X} \cong \ulX \oplus \tau \ulX  \oplus \underline{E}'$.  
Hence we have the equality;
$$
[\ulX, \ulX ] + [\ulX, \tau \ulX ] - [\ulX, \underline{E_X} ] = [\ulX, \ulX ] + [\ulX, \tau \ulX ] - [\ulX , \ulX \oplus \tau \ulX  \oplus \underline{E}']  = -[\ulX , \underline{E}'].
$$   
However, by Proposition \ref{prop of AR triangles}, 
$$0 \geqq -[\ulX , \underline{E}'] = \mu (\ulX, \ulX) + \mu (\ulX, \ulX [-1]) = 1.$$ 
This makes a contradiction, such that $\tau \ulX \cong \ulX [-1]$.

\item[(3)] The condition $(\ast )$ holds when $R = k[[x, y]]/(x^{n+1} + y^2)$ (cf. \cite[Proposition 5.11, Paragraph (9.9)]{Y}). 
\end{itemize}
\end{remark}

We say that $R$ satisfies $(\ast )$ if each AR triangle satisfies the condition $(\ast )$.

\begin{proposition}\label{prop for inequality}
Let $R$ be of finite representation type which satisfies $(\ast )$ and $M$ and $N$ be Cohen-Macaulay $R$-modules with $[\ulM ] = [\ulN ]$ in $K_0 (\SCMR)$.  
Then $\ulM \homo \ulN$ if and only if there exist AR triangles $\underline{Z_i} \to \underline{Y_i} \to \underline{X _i} \to \underline{Z_i}[1]$ and non-negative integers $c_i$ with the equation in $\GSR$;
\begin{equation}\label{expression}
[\ulN ] - [\ulM ] = \sum _{i \in \II} c_i \cdot ([\underline{X_i} ]+ [\underline{Z_i}] - [\underline{Y _i}]), 
\end{equation}
where $\II \subseteq \{ 1, 2, \dots, n \}$. 
\end{proposition}

\begin{proof}
According to the equality (\ref{equation in Grothendieck group}), we see that 
$$
\begin{array}{ll}
[\ulU, \ulN ] - [\ulU, \ulM ] &= \sum _{i \in \II} c_i \cdot ([\ulU, \underline{X_i} ]+ [\ulU, \underline{Z_i}] - [\ulU, \underline{Y _i}]) - \sum _{j \in \JJ}  d_{j} \cdot ([\ulU, \underline{X_{j}}]  + [\ulU, \underline{Z_{j}}] - [\ulU, \underline{Y_{j}}] ) \\
&= \sum _{i \in \II} c_i \cdot \{ \mu (\ulU, \underline{X_i} ) + \mu (\ulU, \underline{X_i}[-1] ) \}  - \sum _{j \in \JJ} d_{j} \cdot \{ \mu (\ulU, \underline{X_{j}})  + \mu ( \ulU, \underline{X_{j}}[-1] ) \}  \\
\end{array}
$$
for each indecomposable module $U \in \CMR$.   
Assume that $\ulM \homo \ulN$. 
If $d_{j} \not= 0$, there exist $i$ such that $\underline{X_{j}}[-1] \cong \underline{X_i}$ with $d_{j} \leq c_i$. 
By the condition $(\ast )$, we can omit such constituents in the expression. 
Repeating this procedure, we obtain the equation. 
\end{proof}

\begin{corollary}\label{cor for inequality}
Let $R$ be of finite representation type which satisfies $(\ast )$. 
Then the stable hom relation is a partial order between modules which give the same class in the Grothendieck group of $\SCMR$. 
\end{corollary}

\begin{proof}
Let $M$ and $N$ be Cohen-Macaulay $R$-modules with $[\ulM ] = [\ulN ]$ in $K_0 (\SCMR)$. 
It follows from Proposition \ref{prop of AR triangles} and the expression in Proposition \ref{prop for inequality} that 
$$
\begin{array}{ll}
[\ulX, \ulN] - [\ulX , \ulM] &= \sum _{i \in \II} c_i \cdot ([\ulX, \underline{X_i} ]+ [\ulX, \underline{Z_i}] - [\ulX, \underline{Y _i}]) \\
&= \sum _{i \in \II} c_i (\mu (\ulX, \underline{X_i}) + \mu (\ulX, \underline{X_i} [-1])).
\end{array}
$$
Suppose that $[\ulX, \ulM] = [\ulX , \ulN]$ for each $\ulX \in \SCMR$ and then $c_i =0$. 
Thus $[\ulM ] - [\ulN ] = 0$ in $\GSR$, so that $\ulM \cong \ulN$. 
\end{proof}

Under the circumstance of Proposition \ref{prop for inequality}, we say that the expression (\ref{expression}) is irredundant if $\underline{X_i} \not\cong \underline{X_j}[-1]$ for each $i$ and $j$ with $i \not= j$. 
Since $R$ is of finite representation type, we can always take the expression irredundant.

\begin{lemma}\label{lemma of stable hom order}
Let $R$ be of finite representation type which satisfies $(\ast )$ and  $M$ and $N$ be Cohen-Macaulay $R$-modules with $\ulM \homo \ulN$ and $[\ulM ] = [\ulN ]$ in $K_0 (\SCMR)$. 
Let $U$ be a non free indecomposable direct summand of $N$. 
Suppose that $[\ulU, \ulM ] = [\ulU , \ulN] $ and $\ulU \cong \ulU [2]$. 
Then $U$ is also a direct summand of $M$.   
\end{lemma}

\begin{proof}
By virtue of Proposition \ref{prop for inequality}, since $\ulM \homo \ulN$, there exist AR triangles $\underline{Z_i} \to \underline{Y _i} \to \underline{X_i} \to \underline{Z_i}[1]$ with 
$$
[\ulN ] - [\ulM ] = \sum _{i \in \II} c_i \cdot ([\underline{X_i} ]+ [\underline{Z_i}] - [\underline{Y _i}])
$$ 
in $\GSR$. 
Thus we have 
$$
\ulN \oplus \bigoplus _{i \in \II} \underline{Y _i}^{c_i} \cong \ulM \oplus \bigoplus _{i \in \II} (\underline{X_i} \oplus \underline{Z_i})^{c_i}. 
$$
Now we assume that $U$ is not a direct summand of $M$. 
Then there exists $i$ such that $\ulU \cong \underline{X_i}$ or $\ulU \cong \underline{Z_i}$ and we can show the inequality
$$
[\ulU, \underline{X_i}] + [\ulU, \underline{Z_i}] - [\ulU, \underline{Y_i}] = \mu (\ulU , \ulX ) + \mu (\ulU, \ulX [-1]) > 0 
$$
holds. 
If $\ulU \cong \underline{X_i}$, it is clear. 
If $\ulU \cong \underline{Z_i}$, since $\ulU$ is periodic of period at most 2, 
$$
\underline{X_i} \cong \tau ^{-1} \ulU \cong \ulU [2-d] \cong \ulU [-d] \cong \ulU \ \text{or} \ \ulU [1] . 
$$
See Remark \ref{Serre functor}. 
Hence we have the inequality above. 
However the inequality never happens since 
$$
[ \ulU , \bigoplus _{i \in \II} \underline{Y_i}^{c_i} ]  =  [\ulU, \bigoplus _{i \in \II} (\underline{X_i} \oplus \underline{Z_i} )^{c_i} ]. 
$$
Therefore, $U$ is a direct summand of $M$. 
\end{proof}

\begin{proposition}\label{proposition of stable hom order}
Let $R$ be of finite representation type which satisfies $(\ast )$ and  $M$ and $N$ be Cohen-Macaulay $R$-modules with $\ulM \homo \ulN$ and $[\ulM ] = [\ulN ]$ in $K_0 (\SCMR)$. 
Suppose that $[\ulU, \ulM ] = [\ulU , \ulN] $ and $\ulU \cong \ulU [2]$ for each indecomposable direct summand $U$ of $N$. 
Then $\ulM \cong \ulN$.    
\end{proposition}

\begin{proof}
For each indecomposable direct summand $U$ of $N$, it follows from Lemma \ref{lemma of stable hom order} that $\ulU$ is isomorphic to a direct summand of $\ulM$, so that $\ulM \cong \ulM ' \oplus \ulU$. 
Set $\ulN \cong \ulN' \oplus \ulU$. 
Then $\ulN' \homo \ulM '$ since $\ulN' \oplus \ulU \homo \ulM ' \oplus \ulU$. 
Note that $[\ulM' ] = [\ulN' ]$ in $K_0 (\SCMR)$ and $[\ulU', \ulM' ] = [\ulU' , \ulN']$ for each indecomposable direct summand $U'$ of $N'$. 
It also follows that $\ulU'$ is isomorphic to a direct summand of $\ulM'$ by Lemma \ref{lemma of stable hom order}. 
Hence, repeating the procedure, we see that $\ulN$ is isomorphic to  a direct summand of $\ulM$. 
Let $\ulM \cong \ulM'' \oplus \ulN$. 
Since $\ulM \homo \ulN$, $\ulM '' \homo \underline{0}$. Thus $\ulM' \cong \underline{0}$. 
Hence, we have $\ulM \cong \ulN$.  
\end{proof}

\section{Stable degenerations of Cohen-Macaulay modules}\label{stable degenerations of Cohen-Macaulay modules}

In this section, we attempt to describe the stable degenerations of Cohen-Macaulay modules in relation with the stable hom relation. 
First let us recall the definition of stable degenerations of Cohen-Macaulay modules. 
For the detail, we refer the reader to \cite{Y11}. See also \cite{SZ14, Y04}.

\begin{definition}\label{stably degeneration}\cite[Definition 4.1]{Y11}
Let $V=k[t]_{(t)}$ and $K = k(t)$. 
For $\ulM, \ulN \in \SCMR$, we say that $\ulM$ stably degenerates to $\ulN$ if there exists a Cohen-Macaulay module ${\underline{Q}} \in \SCM (R\otimes _k V )$ such that ${\underline{Q[1/t]}} \cong {\underline{M \otimes _k K}}$ in $\SCM (R\otimes _k K)$ and ${\underline{Q \otimes _V V/tV}} \cong \ulN$ in $\SCMR$.
\end{definition}

If a ring is an isolated singularity, there is a nice characterization of stable degenerations.

\begin{theorem}\label{conditions for stable degeneration}\cite[Theorem 5.1, 6.1]{Y11}
Consider the following three conditions for Cohen-Macaulay $R$-modules $M$ and $N$: 
 
\begin{itemize}
\item[(1)] $M \oplus P$ degenerates to $N \oplus Q$ for some free $R$-modules $P$, $Q$. 
\item[(2)] $\ulM$ stably degenerates to $\ulN$. 
\item[(3)] There is a triangle 
$$
\begin{CD}
\ulZ @>>> \ulM \oplus \ulZ \ @>>> \ulN @>>> \ulZ [1] \\ 
\end{CD}
$$ 
in $\SCMR$. 
\end{itemize}
If $R$ is an isolated singularity, then $(2)$ and $(3)$ are equivalent. 
Moreover, if $R$ is Artinian, the conditions $(1), (2)$ and $(3)$ are equivalent. 
\end{theorem}

\begin{remark}\label{rmks on stable degenerations}
In general, the implications $(1) \Rightarrow (2)\Rightarrow (3)$ hold and it is required that the endmorphism of $\ulZ$ in the triangle in (3) is nilpotent. 
However if $R$ is an isolated singularity, we do not need the nilpotency assumption (cf. \cite[Lemma 6.5.]{Y11}). 
It follows from the theorem that $\ulM$ and $\ulN$ give the same class in the Grothendieck group of $\SCMR$ if $\ulM$ stably degenerates to $\ulN$. 
\end{remark}

We state order relations with respect to stable degenerations and triangles.

\begin{definition}\cite[Definition 3.2., 3.3.]{HY13}\label{order}
Let $M$ and $N$ be Cohen-Macaulay $R$-modules. 

\begin{itemize}
\item [(1)] We denote by $\ulM \sto \ulN$ if $\ulN$ is obtained from $\ulM$ by iterative stable degenerations, i.e.  there is a sequence of Cohen-Macaulay $R$-modules $\underline{L _0}, \underline{L_1}, \ldots , \underline{L_r}$  such that  $\ulM \cong \underline{L_0}$, $\ulN \cong \underline{L_r}$  and  each $\underline{L_{i}}$ stably degenerates to  $\underline{L_{i+1}}$  for $0 \leq i < r$. 

\item[(2)] We say that $\ulM$ stably degenerates  by a triangle to $\ulN$, if there is a triangle of the form ${\underline{U}} \to \ulM \to {\underline{V}} \to {\underline{U}}[1]$ in $\SCMR$ such that ${\underline{U}}\oplus {\underline{V}} \cong \ulN$. 
We also denote by $\ulM \trio \ulN$  if $\ulN$ is obtain from $\ulM$ by iterative stable degenerations by a triangle. 
\end{itemize}
\end{definition}

\begin{remark}\label{sto implies homo}
It was shown in \cite{Y11} that the stable degeneration relation is a partial order. 
Moreover if there is a triangle ${\underline{U}} \to \ulM \to {\underline{V}} \to {\underline{U}}[1]$, then we can show that $\ulM$ stably degenerates to $\ulU \oplus \ulV$ (cf. \cite[Remark 3.4. (2)]{HY13}). 
Hence $\ulM \trio \ulN$ implies $\ulM \sto \ulN$. 
It also follows from Theorem \ref{conditions for stable degeneration} that $\ulM \sto \ulN$ induces that $\ulM \homo \ulN$. 
\end{remark}

In this section, we shall show
  
\begin{theorem}\label{Main Theorem}
Let $R$ be a hypersurface which is of finite representation type and satisfies $(\ast )$. 
Then $\ulM \homo \ulN$ if and only if $\ulM \sto \ulN$ for Cohen-Macaulay $R$-modules $M$ and $N$ with $[\ulM ] = [\ulN ]$ in $K_0 (\SCMR )$. 
\end{theorem}

To show this, we use the stable analogue of arguments in \cite{Z99}. 
\medskip

The lemma below is well known for the case in abelian categories (cf. \cite[Lemma 2.6]{Z99}). 
The same statement follows in an arbitrary $k$-linear triangulated category, not necessary $\SCMR$. 
(The author thanks Yuji Yoshino for telling him this argument.)

\begin{lemma}\label{ladder}
Let 
$$
\begin{CD}
\Sigma _1: N_1 @>{\tiny 
\begin{pmatrix}
f_1 \\
v  \\
\end{pmatrix} 
}>> L_1\oplus N_2 \ @>{\tiny (u, g_1)}>> L_2 @>>> N_1 [1]   \\ 
\end{CD}
$$ 
and 
$$
\begin{CD}
\Sigma _2:M_1 @>{\tiny 
\begin{pmatrix}
f_2 \\
w  \\
\end{pmatrix} 
}>> N_1\oplus M_2 \ @>{\tiny (v, g_2)}>> N_2 @>>> M_1[1] \\ 
\end{CD}
$$ be triangles in a $k$-linear triangulated category.  
Then we also have the following triangle. 
$$
M_1 \to L_1 \oplus M_2 \to L_2 \to M_1 [1]. 
$$
\end{lemma}

\begin{proof}
We consider the following triangle associated with $\Sigma _2$: 
$$
\begin{CD}
M_1 @>{\tiny 
\begin{pmatrix}
-f_1 \circ f_2 \\
f_2  \\
w \\
\end{pmatrix} 
}>> L_1 \oplus N_1\oplus M_2 \ @>{\tiny 
\begin{pmatrix}
1 & f_1 & 0 \\
0 & v & g_2 \\
\end{pmatrix}
}>> N_2 @>>> M_1[1] . \\ 
\end{CD}
$$
We remark that the left morphism is given by  
$$
\begin{pmatrix}
-f_1 \circ f_2 \\
f_2  \\
w \\
\end{pmatrix} 
= \begin{pmatrix}
1 & -f_1 & 0 \\
0 & 1 & 0 \\
0 & 0 & 1 \\
\end{pmatrix}
\begin{pmatrix}
0 \\
f_2 \\
w 
\end{pmatrix},
$$  
to make the diagram below: 
$$
\begin{CD}
 N_1 @>{\tiny 
\begin{pmatrix}
f_1 \\
v  \\
\end{pmatrix} 
}>> L_1\oplus N_2 \ @>{\tiny (u, g_1)}>> L_2 @>>> N_1 [1]   \\ 
@| @AA{\tiny 
\begin{pmatrix}
1 & f_1 & 0 \\
0 & v & g_2 \\
\end{pmatrix}}A @ AAA  @| \\ 
N_1 @>>{\tiny \begin{pmatrix} 0 \\ 1 \\ 0 \end{pmatrix}}> L_1 \oplus N_1\oplus M_2 \ @>>> L_1 \oplus M_2 @>>> N_1 [1] \\
@. @AA{\tiny \begin{pmatrix}
-f_1 \circ f_2 \\
f_2  \\
w \\
\end{pmatrix} 
}A @AAA @. @. \\
@. M_1 @= M_1 . @. @.
\end{CD}
$$ 
By the octahedral axiom, we obtain the required triangle.  
\end{proof}

\begin{remark}\label{cal of stable hom order on A_n}
Combining (the abelian version of) Lemma \ref{ladder} with Lemma \ref{filtration}, the dimension of $\SHom$ can be calculated easily from the datum of AR sequences. 
For instance, let $R = k[[x, y ]]/(x^{n+1} + y^2)$ where $n$ is even. 
As stated in \cite[Proposition 5.11]{Y}, the set of ideals of $R$ $\{ \ I_i = (x^i , y)\ | \ 1 \leq i \leq n/2 \ \}$ is a complete list of isomorphic classes of indecomposable non free Cohen-Macaulay $R$-modules. 
The AR sequences are 
$$
0 \to I_{i} \to I_{i-1} \oplus I_{i+1} \to I_{i} \to 0 
$$ 
for $i = 1, \cdots , n/2$ where $I_0 = R$ and $I_{n/2 + 1} \cong I_{n/2}$. 
Then we have
$$
\begin{CD} 
I_1 @>>> I_2 @>>> \cdots @>>> I_{n/2 }@>>> I_{n/2} @>>> \cdots @>>> I_1 @>>> R  \\
@VVV @VVV @. @VVV @VVV @.  @VVV @VVV \\
R @>>> I_1 @>>> \cdots @>>> I_{n/2 - 1} @>>> I_{n/2 + 1} \cong I_{n/2} @>>> \cdots @>>> I_2 @>>> I_1.  \\
\end{CD} 
$$
The diagram shows that 
$$
[\SHom _R (-, {I_1})] = \sum _{i=1} ^{n/2} 2 [S_{I_i}]
$$
in $K_0 (\SmodC)$. Thus we obtain $[\underline{I_i}, \underline{I_1}] = 2$ for $i = 1, \cdots , n/2$. 
\end{remark}

\begin{definition}\label{function}
Let $M$ and $N$ be Cohen-Macaulay $R$-modules. We define a function $\uldeltaMN (-)$ on $\SCMR$ by
$$
\uldeltaMN (-) = [ -, \ulN] - [-, \ulM]. 
$$ 
For a triangle $\ulSigma : \ulL \to \ulM \to \ulN \to \ulL [1]$, we also define a function $\delta _{\ulSigma} (-) $ on $\SCMR$ by 
$$
\delta _{\ulSigma} (-) = [ - , \ulL ] + [ - , \ulN ] - [ - , \ulM ] . 
$$
\end{definition}

\begin{remark}\label{rmk for inequality}
As shown in Proposition \ref{prop for inequality}, for modules $M$, $N \in \CMR$ with $\ulM \homo \ulN$ and $[\ulM ] = [\ulN ]$ in $K_0 (\SCMR)$, we have the irredundant expression 
$$
[\ulN ] - [\ulM ] = \sum _{i \in \II } c_i \cdot ([\underline{X_i} ]+ [\underline{Z_i}] - [\underline{Y _i}]),   
$$
where $\underline{\Sigma _{X_i}}: \underline{Z_i} \to \underline{Y_i} \to \underline{X_i} \to \underline{Z_i}[1]$ are AR triangles. 
On the number $c_i$, we have 
$$
c_i = \uldeltaMN (\underline{X_i} ) /\delta_{\underline{\Sigma _{X_i}}}(\underline{X_i}). 
$$
\end{remark}

In the remaining results of the paper, we assume that each indecomposable Cohen-Macaulay module is periodic of period at most 2. 
Then, by \cite[Corollary 6.2]{Eis80}, we see that $R$ must be a hypersurface since $R$ is complete. 
Therefore, in the rest of the paper, we always assume that $R$ is a hypersurface.

We say that a triangle $\ulZ \xrightarrow{\ulg} \ulY \xrightarrow{\ulf} \ulX \to \ulZ [1]$ is without isomorphisms if $\ulf \in \underline{\mathrm{rad}} (\ulY , \ulX)$ and $\ulg \in \underline{\mathrm{rad}} (\ulZ , \ulY)$.  
Let $\ulZ  \to \ulY  \to \ulX  \to \ulZ [1]$ be any triangle. 
As in the case of a sequence, there is a triangle without isomorphisms $\ulZ ' \to \ulY ' \to \ulX ' \to \ulZ '[1]$ such that $\ulZ \cong \ulZ ' \oplus \ulU$, $\ulY \cong \ulY ' \oplus \ulU \oplus \ulV$ and $\ulX \cong \ulX ' \oplus \ulV$ for some $\ulU$, $\ulV \in \SCMR$. (Cf. \cite[Paragraph (2.7)]{Z99}).

\begin{lemma}\label{key lemma}\cite[Lemma 3.1.]{Z99}
Let $R$ be a hypersurface which is of finite representation type and satisfies $(\ast )$. 
Let $M$ and $N$ be Cohen-Macaulay $R$-modules with $\ulM \homo \ulN$ and $[\ulM ] = [\ulN ]$ in $K_0 (\SCMR )$ and let $\ulSigma :  \ulU \to \ulW \to \ulV \to \ulU [1]$ be a triangle without isomorphisms such that $\delta _{\ulSigma} \leq \uldeltaMN$. 
Then there exists a triangle $\underline{\Phi } :  \ulZ \to \ulY \to \ulV \to \ulZ [1]$ without isomorphisms such that $\delta _{\underline{\Phi }} (\ulY ) = \uldeltaMN (\ulY)$. 
\end{lemma}

\begin{proof}
If $\delta _{\ulSigma}(\ulW ) = \uldeltaMN (\ulW )$, we have nothing to prove. 
Otherwise, we assume that there exists an indecomposable direct summand $W_1$ of $W$ ($= W_1 \oplus W_2$) such that $\delta_{\ulSigma}(\underline{W_1} ) < \uldeltaMN(\underline{W_1 })$.  
Under the assumptions, we have AR triangles $\underline{\Sigma _{X_i}}: \tau \underline{X_i} \to  \underline{E_{X_i}} \to \underline{X_i} \to \tau \underline{X_i} [1]$ such that $\uldeltaMN = \sum _i  c_i \cdot \delta_{\underline{\Sigma _{X_i}}}$. (See Remark \ref{rmk for inequality}.)
Thus there exists $i$ such that $\delta_{\underline{\Sigma _{X_i}}}(\underline{W_1}) > 0$. 
This yields that $\underline{W_1} \cong \underline{X_i} $ or $\underline{W_1} \cong \underline{X_i}[-1]$. 
Since each $\ulX \in \SCMR$ is periodic of period at most 2, by Proposition \ref{prop of AR triangles}, $\delta_{\underline{\Sigma _{X_i}}} = \delta_{\underline{\Sigma _{X_i [-1]}}}$. 
Hence we have $\delta_{\underline{\Sigma _{X_i}}} = \delta_{\underline{\Sigma _{W_1}}}$ for the AR triangle $\underline{\Sigma _{W_1}}$ of $\underline{W_1}$.

Let $\ulf$ be the morphism $\ulU \to \underline{W_1}$ in the triangle $\ulSigma$. 
Take the AR triangle $\underline{\Sigma _{W_1}}$ of $\underline{W_1}$ and construct a pullback diagram:
$$
\begin{CD} 
\tau \underline{W_1} @>>> \underline{E_{W_1}} @>>> \underline{W_1}@>>> \tau \underline{W_1} [1]  \\
@| @AAA @AA{\ulf}A @| \\
\tau \underline{W_1}@>>> \underline{E} @>>> \ulU  @>>> \tau \underline{W_1} [1]. \\
\end{CD} 
$$
Since $\ulSigma$ is without isomorphisms, $\ulf$ is not isomorphism. 
By the property of an AR triangle (ART3), the bottom triangle splits, so that $\underline{E} \cong \ulU \oplus \tau \underline{W_1}$. 
Then we have a new triangle:
$$
\underline{\Psi }:  \ulU \oplus \tau \underline{W_1} \to \underline{E_{W_1}} \oplus \ulU \to \underline{W_1} \to (\ulU \oplus \tau \underline{W_1})[1].
$$ 
Applying Lemma \ref{ladder} to the triangles $\ulSigma$ and $\underline{\Psi}$, we get 
$$
\underline{\Theta }: \ulU \oplus \tau \underline{W_1} \to \underline{W_2 } \oplus \underline{E_{W_1}}  \to \ulV \to (\ulU \oplus \tau \underline{W_1}) [1]. 
$$
It is easy to see that we have the following equality
$$
\delta _{\underline{\Theta }}(\ulX)  = \delta _{\ulSigma} (\ulX) + \delta _{\underline{\Psi}} (\ulX) = \delta _{\ulSigma} (\ulX) + \delta _{\underline{\Sigma _{W_1}}} (\ulX ) = \delta _{\ulSigma} (\ulX) + \delta_{\underline{\Sigma _{X_i}}} (\ulX) 
$$
for each $\ulX \in \SCMR$. 
Therefore $\delta _{\ulSigma} < \delta _{\underline{\Theta }} \leq \uldeltaMN$. 
Repeating this procedure, we obtain the required triangle noting that this process stops since $R$ is of finite representation type.   
\end{proof}

\begin{proof}[Proof of Theorem \ref{Main Theorem}]
As mentioned in Remark \ref{sto implies homo}, $\ulM \sto \ulN$ implies that $\ulM \homo \ulN$.

To show the converse, we assume that $\ulM \homo \ulN$ and $\ulM$ and $\ulN$ have no common non-zero direct summand. 
For each indecomposable Cohen-Macaulay module $X$, we set $r(\ulX) = min \{ \uldeltaMN ' (\ulX ), \mu (\ulX, \ulN) \}$ where $\uldeltaMN ' (\ulX ) := \uldeltaMN (\ulX )/ \delta _{\ARTX}(\ulX )$. 
We consider the following set of isomorphism classes of (non free) indecomposable Cohen-Macaulay modules: 
$$
\FF = \{ \ulX \ | r(\ulX) > 0 \} / \cong .
$$
Now we consider a subset $\GG$ of $\FF$ consisting of modules chosen by the following manner. 

\begin{itemize}
\item If $\ulX \in \FF$ and  $\ulX [-1] \not \in \FF$, $\ulX$ belongs to $\GG$. 
\item Assume that $\ulX$ is such that both $\ulX$ and $\ulX [-1]$ belong to $\FF$. 
\begin{itemize}
\item If $r(\ulX [-1] ) \leq r(\ulX )$, $\ulX$ belongs to $\GG$.
\item If $r(\ulX) < r(\ulX [-1])$, $\ulX [-1]$ belongs to $\GG$. 
\end{itemize}
\end{itemize}

Let $N_1 = \bigoplus _{\ulX \in \GG } X ^{r(\ulX)}$, $N_2 = \bigoplus _{\ulX \in \GG} X ^{\mu (\ulX, \ulN) - r(\ulX)}$ and $N_3 = \bigoplus _{\ulX \not\in \GG} X ^{\mu (\ulX, \ulN)}$. 
Then $\underline{N_1} \oplus \underline{N_2} \oplus \underline{N_3} \cong \ulN$ in $\SCMR$. 
Since $R$ is an isolated singularity, we can take an Auslander-Reiten triangle: 
$$
\ARTX : \tau \ulX  \to \underline{E_X} \to \ulX \to \tau \ulX [1].  
$$
Consider a triangle $\ulSigma$ which is a direct sum of $r(\ulX)$ copies of triangles $\underline{\Sigma _{X}}$ where $X$ runs through all modules in $\GG$: 
$$
\ulSigma  : \bigoplus _{\ulX \in \GG} (\tau \ulX )^{r(\ulX)} \to \bigoplus _{\ulX \in \GG}  \underline{E_X} ^{r(\ulX)} \to \bigoplus _{\ulX \in \GG} \ulX ^{r(\ulX)} \to \bigoplus _{\ulX \in \GG} (\tau \ulX [1])^{r(\ulX)}. 
$$
Here we note that  $\bigoplus _{\ulX \in \GG} \ulX ^{r(\ulX)} \cong \underline{N_1}$. 
First we claim that 
\medskip

\noindent
{\it
Claim 1: $\ulN$ is isomorphic to a direct summand of $\bigoplus _{\ulX \in \GG} (\ulX \oplus \tau \ulX )^{r(\ulX) }$.
}
\medskip

According to Proposition \ref{prop for inequality}, we have the following irredundant expression in $\GSR$: 
\begin{equation}\label{expression 2}
[\ulN ] - [\ulM ] = \sum _{i \in \II} \uldeltaMN ' (\underline{X_i}) \cdot ([\underline{X_i} ]+ [\tau \underline{X_i}] - [\underline{E_{X_i}}]). 
\end{equation} 
Since $M$ and $N$ have no common direct summand, $\ulN$ is isomorphic to a direct summand of $\bigoplus _{i \in \II} (\underline{X_i} \oplus \tau \underline{X_i} )^{\uldeltaMN ' (\underline{X_i} )}$.  
Thus, for each indecomposable direct summand $\ulN '$ of $\ulN$, there exists $i$ such that $\ulN' \cong \underline{X_i}$ or $\tau \underline{X_i}$. 
Since each indecomposable Cohen-Macaulay $R$-module is periodic of period at most 2, $\tau \underline{X_i} \cong \underline{X_i}$ or $\underline{X_i} [-1]$ (Remark \ref{Serre functor}). 
Combining Proposition \ref{prop of AR triangles} with Remark \ref{rmk for inequality}, we also have $\uldeltaMN' (\underline{X_i} ) = \uldeltaMN' (\underline{X_i} [-1])$.  
This yields that $r(\ulN ')>0$. 
Hence $\ulN' \in \FF$.

If $\ulN ' \not\cong \tau \ulN '$ then $\mu (\ulN', \ulN ) \leq \uldeltaMN ' (\ulN' )$, so that $\mu (\ulN ', \ulN) = r(\ulN' )$. 
If $\ulN ' \cong \tau \ulN '$ then $\mu (\ulN ', \ulN ) \leq 2\uldeltaMN ' (\ulN ' )$, so that $\mu (\ulN ', \ulN) \leq 2 r(\ulN' )$.  
Hence, if $\ulN' \in \GG$, $\ulN ^{'\mu (\ulN ', \ulN )}$ is isomorphic to a direct summand of $\bigoplus _{\ulX \in \GG} (\ulX \oplus \tau \ulX )^{r(\ulX) }$. 
Suppose that $\ulN' \not \in \GG$. Then $\ulN' [-1]$ belongs to $\GG$ by the definition of $\GG$.   
In this case $r(\ulN' ) < r(\ulN ' [-1])$. 
As mentioned in Remark \ref{the condition for Grothendieck group}(2), $\underline{\tau} (\ulN' [-1] ) \cong (\ulN' [-1])[-1] \cong \ulN' [-2] \cong \ulN'$. 
Thus $\ulN ^{'\mu (\ulN ', \ulN )}$ is isomorphic to a direct summand of $(\ulN' [-1] \oplus \underline{\tau} \ulN' [-1]) ^{r(\ulN ' [-1])}$. 
Hence it is also isomorphic to a direct summand of $\bigoplus _{\ulX \in \GG} (\ulX \oplus \tau \ulX )^{r(\ulX) }$. 
Consequently, the claim holds.  
\qed
\medskip

For the triangle $\ulSigma$, we have an inequality $\delta _{\ulSigma} (\ulX ) = \sum _{\ulX \in \GG} r(\ulX ) \delta _{\ARTX} (\ulX ) \leq \uldeltaMN (\ulX )$ for each $X \in \CMR$. 
By virtue of Lemma \ref{key lemma}, we have a triangle $\underline{\Phi } : \ulZ \to \ulY \to \underline{N_1} \to \ulZ [1]$ such that $\delta_{\underline{\Sigma }} \leq \delta_{\underline{\phi}} \leq \uldeltaMN$ and $\delta_{\underline{\Phi}}(\ulY ) = \uldeltaMN (\ulY )$. 
Next we claim that 
\medskip

\noindent
{\it
Claim 2: $\delta_{\underline{\Phi }}(\ulU ) = \uldeltaMN (\ulU )$ for each $U \in \CMR$. 
}
\medskip

Note that $\sum _{\ulX \in \GG} r(\ulX ) ([\ulX ] + [\tau \ulX ] - [{\underline{E_X} }] )$ is a constituent of (\ref{expression 2}) since $\uldeltaMN ' (\ulX ) \geq r(\ulX ) > 0$ for each $\ulX \in \GG$. 
Seeing the proof of Lemma \ref{key lemma}, $[\underline{N_1}] + [\ulZ ] - [\ulY ]$ can be also taken as one, say 
$$
[\ulN ] - [\ulM ] = [\underline{N_1}] + [\ulZ ] - [\ulY ] + \sum _{i ' \in \II' } \uldeltaMN ' (X _{i'} ) \cdot ([\underline{X _{i'}} ]+ [\tau \underline{X _{i'}}] - [\underline{E_{X_{i'}}}])
$$
in $\GSR$. 
Since $\delta_{\underline{\Phi}}(\ulY ) = \uldeltaMN (\ulY )$ for each indecomposable direct summand $Y'$ of $Y$, $\underline{Y' } \not\cong \underline{X _{i'}}$ and $\underline{X _{i'}}[-1]$. 
Moreover $\underline{Y' } \not\cong \tau \underline{X _{i'}}$ and $\tau \underline{X _{i'}}[-1]$ since $ \tau \underline{X _{i'}} \cong \underline{X _{i'}}$ or $\underline{X _{i'}} [-1]$.  
This implies that $\ulY$ and $\bigoplus _{i' \in \II '}(\underline{X _{i'}} \oplus \tau \underline{X _{i'}})^{\uldeltaMN ' (\underline{X _{i'}})}$ have no common direct summand.  
We remark that 
$$
\ulN \oplus \ulY \oplus \bigoplus _{i' \in \II '} \underline{E_{X_{i'}}}^{\uldeltaMN ' (\underline{X _{i'}} )} \cong \ulM \oplus \underline{N_1} \oplus \ulZ \oplus \bigoplus _{i' \in \II'} (\underline{X _{i'}} \oplus \tau \underline{X _{i'}})^{\uldeltaMN ' (\underline{X _{i'}} )}. 
$$
By the construction of $\underline{\Phi }$, $\ulZ$ contains $\bigoplus _{\ulX \in \GG} \tau \ulX$ as a direct summand. 
Thus, by {\it Claim 1}, one can show that $\ulN$ is isomorphic to a direct summand of $\underline{N_1 } \oplus \underline{Z}$. 
Hence, $\bigoplus _{i' \in  \II'} (\underline{X _{i'}} \oplus \tau \underline{X _{i'}})^{\uldeltaMN ' (\underline{X _{i'}} )}$ is isomorphic to a direct summand of $\bigoplus _{i' \in  \II'} \underline{E_{X_{i'}}}^{\uldeltaMN ' (\underline{X _{i'}} )}$. 
By using the same arguments in Remark \ref{the condition for Grothendieck group}(2), we have
$$
\bigoplus _{i' \in  \II'} \underline{E_{X_{i'}}}^{\uldeltaMN ' (\underline{X _{i'}} )} \cong \bigoplus _{i' \in  \II'} (\underline{X _{i'}} \oplus \tau \underline{X _{i'}})^{\uldeltaMN ' (\underline{X _{i'}} )}. 
$$
Hence $\ulN \oplus \ulY \cong \ulM \oplus \underline{N_1} \oplus \ulZ$, so that $\delta_{\underline{\Phi }} = \uldeltaMN$. 
\qed
\medskip

Since $0 \leq \uldeltaMN - \delta_{\underline{\Phi }} = \delta _{\ulM \oplus \ulZ, \underline{N_2} \oplus \underline{N_3} \oplus \ulY}$, we see that $\ulM \oplus \ulZ \homo \underline{N_2} \oplus \underline{N_3} \oplus \ulY$. 
Moreover {\it Claim 2} implies that $\delta _{\ulM \oplus \ulZ, \underline{N_2} \oplus \underline{N_3} \oplus \ulY} (\underline{N_2} \oplus \underline{N_3} \oplus \ulY) =  \uldeltaMN (\underline{N_2} \oplus \underline{N_3} \oplus \ulY) - \delta_{\underline{\Phi }}(\underline{N_2} \oplus \underline{N_3} \oplus \ulY) = 0$.
By Proposition \ref{proposition of stable hom order}, we have $\ulM \oplus \ulZ \cong \underline{N_2} \oplus \underline{N_3} \oplus \ulY$. 
Therefore, the triangle $\underline{\Phi } : \ulZ \to \ulY \to \underline{N_1} \to \ulZ [1]$ induces a triangle 
$$
\ulZ \to \underline{N_2} \oplus \underline{N_3} \oplus \ulY \cong \ulM \oplus \ulZ \to \underline{N_1} \oplus \underline{N_2} \oplus \underline{N_3} \cong \ulN \to \ulZ [1].
$$
This makes the stable degeneration $\ulM \sto \ulN$.  
\end{proof}
\begin{remark}
It is known that hypersurfaces which are of finite representation type are simple singularities of type $(A_n )$, $(D_n )$, $(E_6 )$, $(E_7)$ or $(E_8)$. 
The Cohen-Macaulay modules are classified and the AR quivers are also described (cf. \cite{Y, KST07}). 
By the classification theorem, we obtain the list of singularities which satisfy the assumption of Theorem \ref{Main Theorem}:  
\begin{center}
\begin{tabular}{|c|c|}\hline
dimension & singularities \\ \hline
odd & $A_n $ \\ \hline
even &$D_{2n} $, $E_7$, $E_8$\\ \hline
\end{tabular}
\end{center}
\end{remark}

We end this paper by giving the description of stable degenerations of Cohen-Macaulay modules over simple singularities of type ($A_n $).

\begin{example}\label{odd simple singularity}
Let $R = k[[x, y ]]/(x^{n+1} + y^2)$. 
Since the stable hom order coincides with the stable degeneration order, we have the following. 
\begin{itemize}
\item[(1)] If $n$ is an even integer, 
$$
\underline{0} \sto \underline{I_1} \sto \underline{I_2} \sto \cdots \sto \underline{I_{n/2}}.  
$$
\item[(2)] If $n$ is an odd integer, 
$$
\underline{0} \sto \underline{I_1} \sto \underline{I_2} \sto \cdots \sto \underline{I_{(n-1)/2}} \sto \underline{N_+} \oplus \underline{N_-}.  
$$
and
$$
\underline{N_{\pm}} \sto \underline{N_{\pm}} \oplus \underline{I_1} \sto \cdots \sto \underline{N_{\pm}} \oplus \underline{I_{(n-1)/2}} \sto \underline{N_{\pm}} \oplus \underline{N_+} \oplus \underline{N_-} \quad \text{(double sign corresponds)}.    
$$ 
\end{itemize}   
See also Proposition \ref{stable hom order on A_n} and Remark \ref{cal of stable hom order on A_n}.  
\end{example}

On Example \ref{odd simple singularity}, the author also investigate the case that the dimension is even in \cite{HY13}. 
The essential part is the following.

\begin{proposition}\label{even simple singularity}\cite[Corollary 2.12., Proposition 3.10.]{HY13}
Let $R = k[[x]]/(x^{n+1} )$. 
Then the stable degeneration order coincides with the triangle order on $\SCMR$.
\end{proposition}

\begin{theorem}\label{simple singularity}
Let $R$ be a simple singularity of type ($A_n$). 
For Cohen-Macaulay $R$-modules $M$ and $N$ with $[\underline{M} ] = [\underline{N} ]$ in $K_0 (\underline{\mathrm{CM}}( R ) )$, the following statements hold.
\begin{itemize}
\item[(1)] If $R$ is of odd dimension then the stable degeneration order coincides with the stable hom order. 

\item[(2)] If $R$ is of even dimension, then the stable degeneration order coincides with the triangle order.   

\end{itemize}
\end{theorem}

\begin{proof}
By virtue of Kn\"orrer's periodicity (cf. \cite[Theorem 12.10]{Y}), we have only to deal with the case $\dim R = 1$ to show (1) and the case $\dim R = 0$ to show (2). 
Hence, by Theorem \ref{Main Theorem} and Proposition \ref{even simple singularity}, we obtain the assertion. 
\end{proof}

\section*{Acknowledgments}
The author express his deepest gratitude to Tokuji Araya, Ryo Takahashi and Yuji Yoshino for valuable discussions and helpful comments. 
The author also thank the referee for his/her careful reading and helpful comments that have improved the paper. 

\bibliographystyle{model1-num-names}

\end{document}